\documentclass[12pt]{amsart}
\usepackage{amsthm,amssymb,amsmath,enumerate,graphicx,amstext}
\usepackage{amscd}
\usepackage[utf8]{inputenc}
\usepackage{setspace}
\theoremstyle{plain}
\newtheorem*{theorem*}{Theorem}
\newtheorem*{remark*}{Remark}

\theoremstyle{remark}

\newtheorem{innercustomgeneric}{\customgenericname}
\providecommand{\customgenericname}{}
\newcommand{\newcustomtheorem}[2]{%
  \newenvironment{#1}[1]
  {%
   \renewcommand\customgenericname{#2}%
   \renewcommand\theinnercustomgeneric{##1}%
   \innercustomgeneric
  }
  {\endinnercustomgeneric}
}

\newcustomtheorem{customthm}{Theorem}
\newcustomtheorem{customlemma}{Lemma}

\begin{document}
\vspace*{3cm}
\thispagestyle{empty}

\author{Shira Zerbib}\thanks{This research was partly supported by the New-England Fund, Technion.}

\address{Department of Mathematics,
University of Michigan, Ann Arbor} \email{zerbib@umich.edu}

\begin{abstract}
A well-known conjecture of Vizing \cite{vizing} is that  $\gamma(G \square H) \ge \gamma(G)\gamma(H)$ for any pair of  graphs $G, H$, where $\gamma$ is the domination number and $G \square H$ is the Cartesian product of $G$ and $H$. Suen and Tarr \cite{ST}, improving a result of Clark and Suen \cite{CS}, showed $\gamma(G \square H) \ge \frac{1}{2}\gamma(G)\gamma(H) + \frac{1}{2}\min(\gamma(G),\gamma(H))$. 
We further improve their result by showing $\gamma(G \square H) \ge \frac{1}{2}\gamma(G)\gamma(H) + \frac{1}{2}\max(\gamma(G),\gamma(H)).$ 
\end{abstract}

\title{An improved bound in Vizing's conjecture}
\maketitle

For a simple graph $G=(V(G),E(G))$ and a vertex $v\in V(G)$, denote by $N_G[v]$ the {\em neighborhood} of $v$ in $G$, namely the set $\{v\} \cup \{u\in V(G) \mid uv \in E(G)\}$. We say that $D\subseteq V(G)$ {\em dominates} $G$ if $V(G)= \bigcup_{v\in D}N_G[v]$. The {\em domination number} $\gamma(G)$ is the minimal size of a set dominating $G$. The {\em Cartesian product} $G \square H$ of a pair of graphs $G,H,$ is the graph whose vertex set is $V(G)\times V(H)$, and whose edge set consists of pairs $(u,v)(u',v')$ in which either $u=u' $ and $vv' \in E(H)$ or $v=v'$ and $uu' \in E(G)$.

In 1963 Vizing \cite{vizing} conjectured that $\gamma(G\square H) \ge \gamma(G)\gamma(H)$. Although proven for certain families of graphs and for all pairs of graphs $G,H$ for which $\gamma(G)\le 3$, this conjecture is still wide open. For surveys on Vizing's conjecture and recent results related to it see \cite{survey, survey2}.

In 2000 Clark and Suen \cite{CS} showed that for every pair of graphs $G,H$ we have $\gamma(G\square H) \ge \frac{1}{2}\gamma(G)\gamma(H)$. Suen and Tarr \cite{ST} then improved this result by showing $\gamma(G\square H) \ge \frac{1}{2}\gamma(G)\gamma(H) + \min(\gamma(G),\gamma(H)).$ Here we further improve these results, establishing $\gamma(G \square H) \ge \frac{1}{2}\gamma(G)\gamma(H) + \frac{1}{2}\max(\gamma(G),\gamma(H)).$. In particular, our proof is simpler the the proof of Suen and Tarr in \cite{ST}. 

\begin{theorem*}\label{main1}
For any pair of graphs $G,H$, $$\gamma(G \square H) \ge \frac{1}{2}\gamma(G)\gamma(H) + \frac{1}{2}\max(\gamma(G),\gamma(H)).$$  
\end{theorem*}
\begin{proof}
Suppose that $\gamma(G) \ge \gamma(H)$. Let $D$ be a dominating set for $G \square H$. 
Let $Q \subset V(G)$ be the projection of $D$ onto $V(G)$. Clearly, $Q$ dominates $G$. Let $U$ be a subset of $Q$ of minimal size that dominates $G$. 
Then $U= \{u_1,\dots,u_k\} \subseteq V(G)$ for some $k \ge \gamma(G)$. Define $S_i=(\{u_i\} \times V(H)) \cap D$, and let $T_i$ the projection of $S_i$ onto $H$. Since $U \subset Q$, for every $1\le i \le k$ we have that $|T_i|=|S_i| \ge 1$. 

Form a partition $\pi_i, ~i=1,\dots,k,$ of $V(G)$,
so that $u_i\in \pi_i$ and $\pi_i\subseteq N_G[u_i]$ for all $i$. This induces a partition $D_i, ~i=1,\dots,k,$ of $D$, where
$D_i = (\pi_i \times V(H)) \cap D$. Let $P_i$ be the projection of $D_i$ onto $H$. Observe that the set $P_i \cup (V (H) -N_H[P_i])$ dominates $H$, and hence for all $i$, 
$$|V(H)-N_H[P_i]| \ge \gamma(H) -|P_i|.$$ 

For $v \in V(H)$ let
$Q_v =D \cap (V(G)\times \{v\})$ and define 
$$C = \{(i,v)\mid \pi_i\times \{v\} \subset N_{G\square H}[Q_v]\}.$$ Set 
$L_i = \{(i,v)\in C \mid v\in V(H)\}$ and  $R_v = \{(i,v)\in C \mid 1\le i \le k\}.$ Then we have, $|C| = \sum_{i=1}^k |L_i| = \sum_{v\in H}|R_v|$. 

Observe that if $v \in (V(H)- N_H[P_i]) \cup T_i$, then the vertices in $\pi_i \times \{v\}$ are dominated by the vertices in $Q_v$ and therefore $(i,v)\in L_i$. Since the sets $V(H)-N_H[P_i]$ and $T_i$  are disjoint, this implies that $|L_i| \ge |V(H)-N_H[P_i]| + |T_i|$. 

Therefore, 
\begin{equation}
\begin{split} 
|C| = &\sum_{i=1}^k |L_i| \ge \sum_{i=1}^k(|V(H)-N_H[P_i]| + |T_i|)\\ & \ge \sum_{i=1}^k(\gamma(H) -|P_i| + 1) \ge k\gamma(H) - |D| + k.
\end{split} 
\end{equation}

We further claim that 
$|R_v| \le |Q_v|$. 
Indeed, if not then the set 
$$U' = \{u\mid (u,v) \in  Q_v\} \cup \{u_j \mid (j,v) \notin R_v\}$$
is a dominating set of $G$ of cardinality $|U'|<|U|$. Moreover, since the projection of $Q_v$ onto $G$ is a subset of $Q$, we have that $U' \subset U$, and thus we obtain a contradiction to the minimality of $U$.  

Thus we have 
\begin{equation} 
|C| \le \sum_{v \in V(H)} |R_v| \le \sum_{v\in V(H)} |Q_v| = |D|.
\end{equation} 

Combining (1) and (2) together we get
$$2|D| \ge k\gamma(H) + k \ge \gamma(G)\gamma(H) + \gamma(G) = \gamma(G)\gamma(H) + \max(\gamma(G),\gamma(H)),$$ which concludes the proof of the theorem.  
\end{proof} 

\begin{remark*}
{\em It follows from the proof that Vizing's conjecture holds for any pair of graphs $G,H$, for which there exists a minimum dominating set $D$ of $G\square H$ so that the projection of $D$ onto $G$ is a minimal dominating set (with respect to containment). Indeed, if such $D$ exists then $U=Q$, and thus instead of Equation (1) we have 
\begin{equation*}
\begin{split} 
|C| = &\sum_{i=1}^k |L_i| \ge \sum_{i=1}^k(|V(H)-N_H[P_i]| + |T_i|)\\ & \ge \sum_{i=1}^k(\gamma(H) -|P_i| + |D_i|) \ge \gamma(G)\gamma(H).
\end{split} 
\end{equation*}
Combining this with Equation (2) we get the result. 

Unfortunately, there exist pairs of graphs $G,H$ for which such $D$ does not exist. An example of such a pair is $G=H=P_4$, where $P_4$ is a path with $4$ vertices.}
\end{remark*} 

\section*{Acknowledgement}
The author is grateful to Ron Aharoni for many helpful discussions.


\begin{thebibliography}{99}
\bibitem{survey} B. Bre\v{s}ar, P. Dorbex, W. Goddard, B. L. Hartnell, M. A. Henning, S. Klav\v{z}ar, and D. F. Rall, Vizing’s Conjecture: A Survey and Recent Results, {\em Journal of Graph Theory}, 69, 46--76, 2012.

\bibitem{CS} W. E. Clark and S. Suen, An Inequality Related to Vizing’s Conjecture, {\em The Electron. J. of Combin}, 7, Note 4, 2000.

\bibitem{fractional} D. C. Fisher, J. Ryan, G. Domke, and A. Majumdar, Fractional domination of strong direct products.
{\em Discrete Appl. Math.}, 50(1), 89--91, 1994.

\bibitem{survey2} B. Hartnell and D. F. Rall, Domination in Cartesian Products: Vizing's Conjecture, in {\em Domination in Graphs-Advanced Topics}, edited by Haynes et al., 163--189, Marcel Dekker, Inc, New York, 1998.

\bibitem{packing} M. S. Jacobson and L. F. Kinch, On the Domination of the Products of Graphs II: Trees. 
{\em J. Graph Theory}, 10, no. 1, 97--106, 1986. 

\bibitem{ST} S. Suen and J. Tarr, An Improved Inequality Related to Vizing’s Conjecture, {\em The Electron. J. of Combin}, 19, 1, 2012.

\bibitem{vizing} V. G. Vizing, The Cartesian product of graphs, {\em Vy\v{c}isl. Sistemy} 9, 30--43, 1963.
\end{thebibliography}
\end{document}